\documentclass[11pt]{amsart}
\usepackage{amsmath, amsthm, amssymb}
\usepackage{amsmath,amscd}
\usepackage{mathabx}
\usepackage{hyperref}
\usepackage{mathrsfs}
\usepackage{xcolor, changepage} 
\usepackage{graphicx}
\usepackage{titletoc}
\usepackage{enumerate}
\usepackage{accents}

\usepackage{tikz}
\usetikzlibrary{matrix,arrows}
\usetikzlibrary{shapes}
\usetikzlibrary{calc}
\usetikzlibrary{arrows}
\usetikzlibrary{decorations.pathreplacing,decorations.markings}
\usepackage[all]{xy}

\usepackage{tikz-cd}

\usepackage{caption}
\usepackage{subcaption}

\usepackage{geometry}
\theoremstyle{plain}
\newtheorem{theorem}{Theorem}
\newtheorem{proposition}[theorem]{Proposition} 
\newtheorem{lemma}[theorem]{Lemma}
\newtheorem{remark}[theorem]{Remark}
\newtheorem{corollary}[theorem]{Corollary}
\newtheorem{definition}[theorem]{Definition}
\newtheorem{example}[theorem]{Example}

\newtheorem{conj}[theorem]{Conjecture}
\numberwithin{theorem}{section}

\DeclareMathOperator{\Lip}{Lip}
\DeclareMathOperator{\dist}{dist}
\DeclareMathOperator{\diam}{diam}
\DeclareMathOperator{\supp}{supp}
\DeclareMathOperator{\degg}{deg}
\DeclareMathOperator{\width}{width}

\numberwithin{equation}{section}

\title[Relative aspherical conjecture]{Relative aspherical conjecture and higher codimensional obstruction to positive scalar curvature}
\author{Shihang He}
\address{Key Laboratory of Pure and Applied Mathematics, School of Mathematical Sciences,
	Peking University, Beijing, 100871, P. R. China}
\email{hsh0119@pku.edu.cn}

\begin{document}
	
	\maketitle

	\begin{abstract}
		Motivated by the solution of the aspherical conjecture up to dimension 5 \cite{CL20}\cite{Gro20}, we want to study a relative version of the aspherical conjecture. We present a natural condition generalizing the model $X\times\mathbb{T}^k$ to the relative aspherical setting. Such model is closely related to submanifold obstruction of positive scalar curvature (PSC), and would be in similar spirit as \cite{HPS15}\cite{CRZ23} in codim 2 case. In codim 3 and 4, we prove results on how 3-manifold obstructs the existence of PSC under our relative aspherical condition, the proof of which relies on a newly introduced geometric quantity called the {\it spherical width}. This could be regarded as a relative version extension of the aspherical conjecture up to dim 5.
	\end{abstract}

	\section{Introduction}
	
	The topological obstruction to manifolds with positive scalar curvature (PSC) is a central problem in differential geometry and geometric topology. With the use of variational method and index theory, many results in this direction have been established in the past a few decades. One of an important kind of obstruction among these results is the submanifold obstruction. More precisely, it is cared about when a submanifold, of certain topological type and in a suitable position of the ambient space, becomes the PSC obstruction of the ambient manifold. 
	
	The first progress in this direction dates back to the pioneering work of Schoen-Yau \cite{SY79a}, where an {\it incompressible hypersurface obstruction} theorem was established. In fact, it was proved in \cite{SY79a} that if a 3-dimensional compact manifold contains an incompressible surface of positive genus, then this manifold admits no PSC metric. Later, Gromov-Lawson \cite{GL83} generalized this to higher dimension. Recently, Cecchini-R{\"{a}}de-Zeidler proved the following codimension 1 obstruction theorem, which serves as a stability version in the codimension 1 case.
	
	\begin{theorem}\label{thm: CRZ23a}(\cite{CRZ23})
		Let $Y$ be an orientable connected n-dimensional manifold with $n\le 7, n\ne 5$ and
		let $X\subset Y$ be a two-sided closed connected incompressible hypersurface which admits no PSC metric. Suppose that one of the following two conditions holds in the case $n\ge6$:
		
		(a) $Y$ is almost spin.
		
		(b) $X$ is totally nonspin.
		
		Then $Y$ admits no complete PSC metric.
	\end{theorem}
	Theorem \ref{thm: CRZ23a} is an important generalization of Schoen-Yau's result. Codimension 1 results of similar version has also be studied by various authors, In \cite{Zei17}, Zeidler established an index theoretic version of this kind of obstruction. Recently, in \cite{CLSZ21}, a similar type result for hypersurface lying in certain manifold class called the $\mathcal{C}_{\degg}$ class has also been proved. Notice that the incompressible condition among these results plays crucial role in describing the suitable position of the hypersurface, since without this condition one could easily construct counterexample such that the theorem fail.
	
	In codimension 2, the first study was carried out by Gromov-Lawson \cite{GL83}, where they proved $X\times \mathbb{R}^2$ carries no complete metric with uniformly positive scalar curvature when $X$ is an enlargeable spin manifold. However, to consider more general settings, one could not expect a single incompressible condition be enough in codimension 2 case. In fact, it is obvious that $X$ is incompressible in $\mathbb{S}^2\times X$, but the latter always admits PSC metric. To rule out this case, Hanke-Pape-Schick \cite{HPS15} found a natural condition on the second homotopy group, and by using a theorem in \cite{HS06}, they have generalized the Theorem of \cite{GL83} into the following form:
	
	\begin{theorem}\label{thm: HPS15}(\cite{HPS15})
		Let $Y$ be a closed connected spin manifold. Assume
		that $X\subset Y$ is a codimension two submanifold with trivial normal bundle and that
		
		(1)$\pi_1(X)\longrightarrow\pi_1(Y)$ is injective
		
		(2)$\pi_2(X)\longrightarrow\pi_2(Y)$ is surjective.
		
		Assume that the Rosenberg index of $X$ does not vanish: $0\ne\alpha(X) \in K_{*}(C^*\pi_1(X))$
		
		Then $Y$ admits no PSC metric.
	\end{theorem}
	
	Related stability version of this theorem was also verified by  Cecchini-R{\"{a}}de-Zeidler up to dimension 7.
	
	\begin{theorem}\label{Thm: CRZ23b}(\cite{CRZ23})
		Let $Y$ be a $n$-dimensional closed connected manifold, $n = 3,4,5,7$. Assume
		that $X\subset Y$ is a codimension two submanifold with trivial normal bundle and that
		
		(1)$\pi_1(X)\longrightarrow\pi_1(Y)$ is injective
		
		(2)$\pi_2(X)\longrightarrow\pi_2(Y)$ is surjective.

            Assume $X$ admits no PSC metric. Then $Y$ admits no PSC metric.
	\end{theorem}
	
	Obviously, the condition describing the position of $X$ in $Y$ in Theorem \ref{thm: HPS15} and Theorem \ref{Thm: CRZ23b} is equivalent to the relative homotopy condition: $\pi_2(Y,X) = 0$.
	
	In higher codimension, it is an interesting problem to ask what is the effect of the submanifold to the PSC obstruction of the ambient space. Here are several examples: the case that $Y$ turns out to be a fiber bundle with fiber $X$ over the base space $B$ provides a special setting of this problem, and this has been studied by Zeidler \cite{Zei17} and the author \cite{He23} by using index theory and variational method respectively. In \cite{WXY21}, a high codimensional cube inequality was established, which describes the effect of the PSC obstruction of the submanifold to the multi-distance spread of the ambient cube-like manifold. This also reflects certain interaction between the submanifold PSC obstruction and the geometry of the ambient space.
	
	Another notable series of results of high codimensional PSC obstruction are obtained by constructing transfer map for certain generalized homology group from the ambient space to the submanifold, and one may see \cite{Eng18}\cite{NSZ21}\cite{Zei17} for progress in this direction. For example, the following was proved in \cite{Zei17} by Zeidler.
	
	\begin{theorem}(\cite{Zei17})
		Let $X$ be a codimension $k$ submanifold in $Y$ with trivial normal bundle, with $\pi_i(Y)=0, i=2,3,\dots,k$. Suppose $\hat{A}(X)\ne 0$ and $\pi_1(Y)$ satisfies the Strong Novikov Conjecture, then $0\ne\alpha(Y) \in K_{*}(C^*\pi_1(Y))$.
	\end{theorem}
	
	These results illustrate how $X$, a codimension $k$ submanifold affect the PSC obstruction for a sufficient connected ambient space $Y$. Intuitively, such requiement for the ambient space is designed to rule out the case of $\mathbb{S}^k\times X^n$. However, to some extent, the sufficient connected condition for $Y$ may give a priori constraint for itself, and it seems not so clear how $Y$ interact with $X$ in this case.
	
	In this work, we hope to find new natural condition which would provide PSC obstruction from high codimensional submanifold. The condition we consider, the {\it relative aspherical condition}, is a {\it homotopical condition} of {\it relative type}, which we think may reflect the interaction of the ambient space and the submanifold in a better way. The definition is as follows:
	
	\begin{definition}\label{defn: relative aspherical}
		
		Let $X^n$ be a submanifold of $Y^{n+k}$. We say that
		
		(1) $Y$ is aspherical relative to $X$, if $\pi_i(Y,X) = 0$ for $i = 2, 3, \dots$.
		
		(2) $Y$ is weakly aspherical relative to $X$, if $\pi_i(Y,X) = 0$ for $i = 2, 3, \dots , k$.
	\end{definition}

    \begin{remark}
        By looking at the long exact sequence of the homotopy group, it is clear that we have the following equivalent definition, which would also be useful sometimes:
        
        (1) $Y$ is aspherical relative to $X$, if

        \quad(a) $\pi_1(X)\longrightarrow \pi_1(Y)$ is injective.

        \quad(b) $\pi_i(X)\longrightarrow\pi_i(Y)$ is an isomorphism for $i\ge 2$.
        
        (2) $Y$ is weakly aspherical relative to $X$, if

        \quad(a) $\pi_1(X)\longrightarrow \pi_1(Y)$ is injective.

        \quad(b) $\pi_i(X)\longrightarrow\pi_i(Y)$ is an isomorphism for $i = 2,\dots,k-1$.

        \quad(c) $\pi_k(X)\longrightarrow\pi_k(Y)$ is surjective.
    \end{remark}
	Compared with earlier results for high codimensional PSC obstruction, since we only concern the behavior of the submanifold relative to the ambient space, we need not to make any a priori assumption on the topology of the ambient space. Moreover, our assumption is topologically intrinsic, and no extra structural or geometric condition is required in this setting. Now let us formulate our {\it relative aspherical conjecture}, stated under condition (1) and (2) in Definition \ref{defn: relative aspherical} respectively:
	
	\begin{conj}\label{conj: RASC}(Full Relative Aspherical Conjecture)
		Let $Y^{n+k}$ be a compact manifold and $X^n$ an codimension $k$  submanifold with trivial normal bundle, such that $Y$ is aspherical relative to $X$, $n\ne 4$. If $X$ admits no PSC metric, then $Y$ admits no PSC metric.
	\end{conj}
	
	\begin{conj}\label{conj: RASC strong}(Strong Relative Aspherical Conjecture)
		Let $Y^{n+k}$ be a compact manifold and $X^n$ an codimension $k$  submanifold with trivial normal bundle, such that $Y$ is weakly aspherical relative to $X$, $n\ne 4$. If $X$ admits no PSC metric, then $Y$ admits no PSC metric.
	\end{conj}
	
	Obviously, since Conjecture \ref{conj: RASC strong} assumes weaker condition, its conclusion would be stronger than Conjecture \ref{conj: RASC}. We remind the readers that Conjecture \ref{conj: RASC strong}, the strong version conjecture is proposed for weakly relative aspherical condition, which only requires vanishing relative homotopy group up to dimension $k$.
	
	At the very begining point, we would like to point out Conjecture \ref{conj: RASC} and \ref{conj: RASC strong} actually generalizes the aspherical conjecture of absolute version, as well as various interesting stability type conjucture {\it into a single setting}. The following examples illustrate this point:
	
	(1) If $X$ is a point, Conjecture \ref{conj: RASC} obviously implies the {\it aspherical conjecture}, which was recently verified in \cite{CL20}\cite{Gro20} up to dimension $5$. If $X$ is $S^1$, then Conjecture \ref{conj: RASC} in this case is also equivalent to the aspherical conjecture, since the fundamental group of a closed aspherical manifold is torsion free. Here one should note that the codimension $k$ relative aspherical conjecture (relative to $S^1$) implies the aspherical conjecture of dimension $k+1$.
	
	(2) Let $Y=X\times S^1$, then Conjecture \ref{conj: RASC} implies the Rosenberg $S^1$ stability conjecture, see \cite{Ros07}\cite{R23}.
	
	(3) In codimension 2, the Conjecture \ref{conj: RASC strong} is true for a large class of manifold, owing to the results of \cite{HPS15} (Theorem \ref{thm: HPS15}) and \cite{CRZ23} (Theorem \ref{Thm: CRZ23b}).
	
	(4) If $E$ is a $F$ bundle over an aspherical manifold, then $E$ is aspherical relative to $F$. One may refer to \cite{Zei17} for related results. As a special case, when $E=F\times B$, $B$ is a closed aspherical manifold, then $E$ is also aspherical relative to $F$.
	
	\quad
	
	Therefore, for further investigation of the interaction of PSC obstruction through high codimension, Conjecture \ref{conj: RASC} and \ref{conj: RASC strong} turn out to be problems worth studying. The case that the $k\ge3$ has not been well understood yet. A difficulty lies in that, even the simplest case that $X=S^1$ and $k=3$ would imply the 4-dimensional aspherical conjecture. Based on Dirac operator method, \cite{Yu98}\cite{Dra06} implies such kind of result for a large class of aspherical manifold, {\it i.e.} those with finite asymptotic dimension for their fundamental group. However, even in dimension four, this has only been settled by minimal hypersurface method in full generality at the present time.
	
	In this paper, at the first stage of attacking Conjecture \ref{conj: RASC} in higher codimension, we shall study the PSC obstruction from enlargeable submanifold via relative aspherical condition in codimension 3 and 4. Our main result states as follows:
	\begin{theorem}\label{thm1}
		Let $Y^{n+k}$ be a compact manifold and $X^n$ a codimension $k$  enlargeable submanifold with trivial normal bundle $(n+k\le 7)$, such that $Y$ is aspherical relative to $X$. Assume one of the following happens:
		
		(a) $k=3$.
		
		(b) $k=4$ and the Hurewicz map $\pi_2(X)\longrightarrow H_2(X)$ is trivial.
		
		Then $Y$ admits no PSC metric. 
	\end{theorem}
	
	One may also expect the conclusion holds true under weakly relative aspherical condition, {\it i.e.}, one may expect results corresponding to the stronger Conjecture \ref{conj: RASC strong}. To this end, we can show the following result:
	\begin{theorem}\label{thm2}
		Under the assumption of Theorem \ref{thm1}, if there exists a closed aspherical, enlargeable manifold $Z$ and a map $\phi:X\longrightarrow Z$ with non-zero degree, then the conclusion of Theorem \ref{thm1} holds true under weakly relative aspherical condition.
	\end{theorem}
	
	For the special case that $X=S^1$, Theorem \ref{thm1} and Theorem \ref{thm2} reduces to the absolute version of the aspherical conjecture up to dimension 5. Now let us get back to Conjecture \ref{conj: RASC} and \ref{conj: RASC strong}. Though it seems hard to confirm them in general cases, Theorem \ref{thm2} already gives the following partial affirmative answer for Conjecture \ref{conj: RASC} up to dimension $7$.
	
	\begin{corollary}\label{cor1}
		The strong Conjecture \ref{conj: RASC strong} holds true in following cases:
		
		(1) $k=3$, $n\le 3$.
		
		(2) $k=4$, $n\le3$, and $X$ contains no $S^2\times S^1$ factor in its prime decomposition when $n=3$.
	\end{corollary}
	
	Finally, let us provide several applications of our main results. The first one of these concerns PSC obstruction for fiber bundle over aspherical space. One may compare this with \cite{Zei17}\cite{He23}.
	\begin{corollary}\label{cor2}
		Let $Y^{n+k}$ be a fiber bundle over a closed aspherical manifold $B^k$ ($k=4,5, n+k\le 7$) with fiber $F$. If $F$ admits no PSC metric, then $Y$ admits no PSC metric.
	\end{corollary}
	
	In particular, we can prove the following codimension 2 obstruction result.
	\begin{corollary}\label{cor3}
		Let $Y^n$ $(n\le 7)$ be a noncompact manifold which contains an embedded, codimension $2$ closed aspherical sumbanifold as a deformation retract, then $Y$ admits no complete metric with uniformly positive scalar curvature.
	\end{corollary}
	
	The following corollary concerns PSC obstruction for sufficiently connected manifold. One may compare this with index theoretic results like \cite{Eng18}\cite{NSZ21}\cite{Zei17}, as well as the classification result in \cite{CLL23}.
	\begin{corollary}\label{cor4}
		Let $Y^n$ $(n\le 7)$ be a closed manifold with $\pi_2(Y)=\pi_3(Y) =\dots= \pi_k(Y) = 0$ $(k=3,4)$, containing an embedded, incompressible, codimension $k$ enlargeable aspherical submanifold. Then $Y$ admits no PSC metric.
	\end{corollary}
	
	The last application concerns the aspherical conjecture in higher dimension. It's of similar spirit to Theorem 7.47 in \cite{GL83}. In  Sec. 7.5 of his four lecture \cite{Gro23}, Gromov has also studied this kind of problem by using very different method. 
	\begin{corollary}\label{cor5}
		Let $Y^n$ $(n\le 7)$ be a closed aspherical manifold such that $\pi_1(Y)$ contains a subgroup isomorphic to that of some codimension $4$ closed smooth aspherical manifold, then $Y$ admits no PSC metric.
	\end{corollary}
 As a result, for $n\le 7$, closed aspherical $n$-manifold with PSC metric does not cotain $\mathbb{Z}^{n-4}$ in its fundamental group. 
	
	Now let us briefly explain the main idea and key observations in the proof of the above theorems. The proof is based on Gromov's $\mu$-bubble \cite{Gro18}\cite{Gro23} in combination with some quantitative topology argument. In codim 3, we must make reduction along $X$, and collect the PSC information on several 2-spheres. Recall in the proof of the aspherical conjecture in \cite{CL20} and \cite{Gro20}, a key step was to obtain the relative filling radius upper bound for certain 2-chain. Unfortunately, this could not be directly applied to our case, where the universal covering of $Y$ may not be contractible, and the homology class represented the sphere may be nontrivial. Therefore, the feasibility of defining filling radius provides essential difficulty. Instead of estimating filling radius, we introduce a quantity which is defined as the homological width in Defnition \ref{Defn: homological width} to represent the minimal diameter of the chain representing certain homology class. In dimension 2, we simply interpret this as the spherical width. In Theorem \ref{thm: sphere width estimate}, we establish a lower bound estimate of this quantity at infinity. As a result, this quantity would be large at infnity, but forced to be small by PSC condition, which gives the contradiction.
	
	Recent years, motivated by the pioneering work of Gromov \cite{Gro18}, the width of the Riemannian band has been studied extensively. See for example, \cite{Zhu21}\cite{Zei20}\cite{Zei22} as well as \cite{CZ21}\cite{GXY20}\cite{WXY21}\cite{R23}\cite{Ku23}. The band width estimate is not only important for people to understand scalar curvature geometry but also useful in yielding topological obstruction to PSC metric. One may sometimes show certain covering of certain manifold contains a long band, and hence admits no PSC metric. Such kind of application actually consolidates the vague philosophy proposed by Gromov \cite{Gro86} that {\it large Riemannian manifold admits no PSC metric}. However, this kind of band argument does not always work in all of the largeness related settings. For instance, this fails in the case of aspherical manifold, since it remains a problem whether all of the aspherical manifolds are enlargeable. Our spherical width could actually be regarded as a high dimensional analogue of the width and turns out to be valid in problems concerning aspherical manifold. Additionally, compared with relative filling radius, it's not sensible to complicated topology of the ambient manifold and could always be defined. The proof of Theorem \ref{thm1} is philosophically clear. Like what was proposed in \cite{Gro86}, the PSC obstruction still lies in the largeness of certain covering space. In the case that band argument turns out to be accessible, the PSC obstruction lies in that the manifold may be wide in certain direction. In our case, it lies in the existence of certain {\it large sphere} at infinity (In fact, by the language of our proof, it is a sphere with non-trivial $\zeta$-image which is also far away from $X_0$).
	
	\quad

	The rest of the paper runs as follows: In Section 2 we recall useful facts and prove several lemmas which would be used later. In Section 3 we collect useful information for our topological setting. In Section 4 we present a systematic discussion to the spherical width and give the proof of Theorem \ref{thm1}. In Section 5 we prove Theorem \ref{thm2}. In the last section we prove the corollaries.
	
	\quad

	\textbf{Acknowledgement} This work is supported by National Key R\&D Program of China Grant 2020YFA0712800. The author would like to express his deepest gratitude to Prof. Yuguang Shi for constant encouragement and support. He would like to thank Dr. Jintian Zhu for inspiring discussions. He is also grateful to Prof. Man Chun Lee for encouragement and enlightening discussion.

	\section{Preliminary}
	In this section, we would recall basic concepts and several important results which would be used in the proof of the main theorem. For some of these, we may make necessary refinement so as to better apply them to the setting we discuss.
	\subsection{Enlargeable manifold}
	
	In this subsection we recall the definition of enlargeable manifold in \cite{GL83}. Note here a difference is that we do not require any spin condition.
	\begin{definition}\label{Defn: enlargeable}
		A compact Riemannian manifold $X^n$ is said to be enlargeable if for each $\epsilon>0$, there exists an oriented covering $\Tilde{X}\longrightarrow X$ and a map $f:\Tilde{X}\longrightarrow S^n$ to the unit sphere in the Euclidean space with non-zero degree, such that $\Lip f<\epsilon$. 
	\end{definition}
	
	The next is a useful property in describing enlargeable manifold.
	\begin{lemma}\label{lem: enlargeable}
		Let $X$ be a compact enlargeable manifold, then for any $d>0$ there exists a covering $\Tilde{X}$ of $X$ and a cube like region $V$ in $\Tilde{X}$, such that
		\begin{align*}
			\dist(\partial_{-i}V,\partial_{+i}V)>d, \mbox{ for } i = 1,2,\dots,n
		\end{align*}
		Here the cube like region means that there exists a non-zero degree map $\varphi:V\longrightarrow [-1,1]^n$, and we denote $\partial_{\pm i}V = \varphi^{-1}(\partial_{\pm i}$).
	\end{lemma}

 \subsection{Counting the intersection number}
We record the following general lemma which would be useful in counting intersection number. For a proof the reader may refer (\cite{Fr23}, Theorem 147.5).

\begin{lemma}\label{lem: counting}
    Let M be a compact oriented $m$-dimensional smooth manifold together with a boundary decomposition
$\partial M=A\cup B$. Let $X$ and $Y$ be a complementary pair of oriented submanifolds of $M$ with $\partial X\subset A$ and $\partial Y\subset B$ intersecting transversally. We write $k = dim(X)$ and we denote by $i: X\longrightarrow M$ and $j : Y \longrightarrow M$ the obvious inclusion maps. Furthermore we denote by $[X] \in H_k(X, \partial X)$ and $[Y] \in H_{m-k}(Y, \partial Y)$ the fundamental classes of $X$ and $Y$. Then the oriented intersection number of $X$ and $Y$ equals
\begin{align*}
    \langle D_M(i_*[X])\smallsmile D_M(j_*[Y]), [M]\rangle = \langle D_M(i_*[X]),j_*[Y]\rangle \in \mathbb{Z}
\end{align*}
Here $[M]\in H_m(M,\partial M)$ denotes the fundamental class of $M$ and $D_M$ denotes the  Poincare-Lefschetz duality map $H_k(M,A)\longrightarrow H^{m-k}(M,B)$ or $H_{m-k}(M,B)\longrightarrow H^k(M,A)$.
\end{lemma}

	\subsection{Filling estimate and slice and dice}
	In this subsection we recall important elements used in the proof of the aspherical conjecture up to dimension 5 in \cite{CL20}\cite{Gro20}. The first lemma focus on filling of chain. Since later we have to apply it to non-contractible space, we have made necessary refinements on the original filling estimate in \cite{CL20}\cite{Gro20}. 
	\begin{lemma}\label{lem: HZ23}
		(\cite{HZ23}, Lemma 2.1)
		Let $\pi:(\Tilde{Y}^{n},\tilde g)\longrightarrow (Y^{n},g)$ be a Riemannian covering of the compact manifold $Y$. Then for any $r>0$ there is a constant $R=R(r)>0$ with the property that for any $i$-dimensional boundary $\alpha$ in $\tilde{Y}$ with $\diam(\alpha)\le r$, there is a $i+1$-chain $\beta$ in $\Tilde{Y}$ with $\diam(\beta)\le R$ and $\partial\beta = \alpha$. Here we use $diam(\cdot)$ to denote the diameter of the support of the chain.
	\end{lemma}

        \begin{proof}
            Fix a point $q\in\Tilde{Y}$ and a point $p$ in the support of $\alpha$. We can always find a Deck transformation $\Phi$ such that $d(\Phi(p),q)\le D:= \diam(Y)$. As a result, $\Phi_\#(\alpha)$ is supported in the ball $ B_q(r+D)$. Denote $\mathcal K = \ker (H_i(B_q(r+D))\to H_k(\tilde Y))$, then $\Phi_\#(\alpha)$ lies in $\mathcal{K}$ because it is a boundary. Since $\mathcal{K}$ is finitely generated, we can find a positive constant $R$ such that $\mathcal K\to H_i(B_q(R))$ is a zero map. This yields that $\Phi_\#(\alpha)$ can be filled by a chain of diameter no greater than $R$ and the same thing also holds for $\alpha$. This completes the proof of Lemma \ref{lem: HZ23}.
        \end{proof}
	
	The next lemma is a slight refiment of the slice and dice procedure developed in \cite{CL20}.
	
	\begin{lemma}\label{lem: slice and dice}
		
		Let $\Gamma^l$ be a closed connected Riemannian manifold with $\mathbb{T}^N$-stabilized scalar curvature $Sc_N^{\rtimes}(\Gamma)\ge 1$, $l+N\le 7$.
		
		(1) If $l=2$, then $\Gamma$ is homeomorphic to a sphere, and there exists a universal constant $L_0$, such that.
		
		\begin{align*}
			\diam(\Gamma)\le L_0
		\end{align*}
		
		(2) If $l=3$, then $\Gamma$ can be divided into regions $U_i (i = 1,2,\dots,u)$ by a collection of mutually disjoint embedded spheres $S_{\alpha}$'s ($\alpha\in A$, here $A$ is a finite index set):
		
		\begin{align*}
			&\Gamma = \bigcup_{i=1}^u U_i\\
			&\partial U_i = \bigcup_{\alpha\in A(i)} S_{\alpha}
		\end{align*}
		
		Furthermore, there exists a universal constant $L_0$, such that
		
		\begin{align*}
			&\diam(U_i)\le L_0\\
			&\diam(S_{\alpha})\le L_0
		\end{align*}
	\end{lemma}
	\begin{proof}
		The lemma follows almost from the argument of \cite{CL20}, but some necessary modification is needed. Let us first recall the slice and dice procedure developed in \cite{CL20}. By looking for minimizing surface for certain weighted area functional on $\Gamma$, one obtains a finite collection of the slicing surface $\mathcal{S}$, each homeomorphic to $S^2$. The consequence is that the first betti number of $M\backslash\mathcal{S}$ vanishes, and one could start the dicing procedure as follows: Fix a point $p$ and some universal constant $L_2$, solving weighted free boundary $\mu$-bubble problem on the Riemannian bands
  \begin{align*}
      \mathcal{V}_j = B_p((j+1)L_2)\backslash B_p(jL_2), j = 0,1,2,\dots, [\diam(\Gamma\backslash\mathcal{S})/L_2]+1
  \end{align*}
  It was shown in \cite{CL20} that the dicing surfaces, {\it i.e.} the topological boundary for the $\mu$-bubbles are either topological disk or topological sphere. Denote $\mathcal{D}_1$ to be the disjoint collection of the dicing surface homeomorphic to disk and $\mathcal{D}_2$ to be the disjoint collection of the dicing surface homeomorphic to sphere. \cite{CL20} has shown $\mathcal{S}$ together with $\mathcal{D}_1$ and $\mathcal{D}_2$ divide $\Gamma$ into regions of uniform diameter bound $L_1$, and that the diameter of each element in $\mathcal{S},\mathcal{D}_1$ and $\mathcal{D}_2$ has diameter bound $L_1$. Morover, one has $\mathcal{S}\cap\mathcal{D}_2\ne \emptyset$.
  
        Now we have to modify surfaces in $\mathcal{D}_1$ to obtain a new class of mutually disjoint spheres. For any $D\in \mathcal{D}_1$, since $\partial D$ is connected, it could touch exactly one slicing surface $S$. Fix this slicing surface, and denote $\lbrace D_1,D_2,\dots,D_t\rbrace$ to be the elements in $\mathcal{D}_1$ touching $S$. We begin by dealing with the case that $t=1$. At this time, by the Jordan curve Theorem, $\partial D_1$ devides $S$ into two parts $D_1^+$ and $D_1^-$. Then $D_1\cup D_1^{\pm}$ turn out to be the boundary of some region obtained in \cite{CL20}, denoted by $U_1^{\pm}$, with $\diam(U_1^{\pm})<L_1$. 

        Let $S\times [0,\epsilon] (\epsilon<1)$ be the small tubular neighbourhood of $S$ and $V = D_1^+\times [0,\epsilon]$. Define
        \begin{align*}
            \hat{U}_1^+ = U_1^+\backslash V,\quad \hat{U}_1^- = U_1^-\cup V
        \end{align*}
        We have
        \begin{align*}
            &\diam(\hat{U}_1^+)\le \diam(U_1^+)<L_1\\
            &\diam(\hat{U}_1^-)\le \diam(U_1^-)+\diam(V)\le \diam(U_1^-)+\diam(S)+\epsilon<2L_1+1
        \end{align*}
        Substitute $D_1$ by the sphere $\hat{D}_1 = D_1\cup D_1^+\times \lbrace\epsilon\rbrace$, we get $\hat{D}_1$ and $S$ successfully separated.

        For the general case that $S$ touches $t$ disks, consider the curves $\partial D_j$ on $S$. One could start with an innermost curve, and assume it is $D_1$ without loss of generality. Then the above procedure applies. By repeating this procedure for $D_j$'s, and at the same time choose $\epsilon$ smaller and smaller, one is able to separate all this disks away from $S$. Finally, we obtain a collection of spheres separating $\Gamma$ into regions with diameter bounded by $L_0=2L_1+1$, and this completes the proof of Lemma \ref{lem: slice and dice}.
	\end{proof}
	
	\subsection{$\mu$-bubble reduction in cubical region}
	
	The cube inequality was first introduced by Gromov in his Four Lecture \cite{Gro23} to describe the distance stretching for certain cubical region in multi-directions. Later, it was studied in \cite{WXY21} a high dimensional version of this inequality in spin setting. In this subsection, we focus on a $\mu$-bubble reduction lemma in cubical region. A detailed proof for similar conclusion has already appeared in \cite{GZ21}. However, for the convenience of the reader, we would like to collect the basic notations and results in this subsection. 
	
	Let $X$ be a compact Riemannian manifold of dimension $n+k\ge 3$ with boundary. We shall divide the boundary of $X$ into two piecewisely smooth parts, the {\it effective boundary} and the {\it side boundary}, such that they have a common boundary in $\partial X$. We denoted this by $\partial X = \partial_{eff}\cup\partial_{side}$. 
	
	Let
	\begin{align*}
		f: (X,\partial_{eff})\longrightarrow([-1,1]^n,\partial[-1,1]^n)
	\end{align*}
	be a continuous map from $X$ to a $n$-cube. In our convention we shall always assume that the effective boundary coincides with the inverse image of the boundary of the cube under the map $f$. Let
	\begin{align*}
		h = f_![t] \in H_k(X,\partial_{side}), t\in int[-1,1]^n
	\end{align*}
	be the point pullback of $f$. Here $f_!$ is the wrong way map between the homology group.
	
	Let $\partial_{-i},\partial_{+i}\subset \partial [-1,1]^n$ be the pair of opposite faces of the cube for $i = 1,2,\dots, n$. We further denote
	\begin{align*}
		&\partial_{-i}X = f^{-1}(\partial_{-i})\\
		&\partial_{+i}X = f^{-1}(\partial_{+i})
	\end{align*}
	to be the portion of $\partial_{eff}$, and
	\begin{align*}
		d_i = dist(\partial_{-i}X, \partial_{+i}X), i = 1,2,\dots, n
	\end{align*}
	to be the distance of the distinguished boundary portion in $X$.
	
	\begin{definition}
		Let $(Y,g)$ be a Riemannian manifold. We say $Y$ has $\mathbb{T}^N$-stabilized scalar curvature at least $\sigma$, if there exists a Riemannian manifold $(Y_N,g_N)$, $Y_N = Y\times \mathbb{T}^N$ and $g_N$ has the following form
		\begin{align*}
			g_N = g+\sum_{i=1}^{N} \varphi_i^2dt_i^2
		\end{align*}
		for some positive smooth function $\varphi_i$ on $Y$, such that
		\begin{align*}
			Sc(g_N)\ge \sigma
		\end{align*}
		We denote this by
		\begin{align*}
			Sc_N^{\rtimes}(Y)\ge \sigma
		\end{align*}
	\end{definition}
	
	Now we can state the $\mu$-bubble reduction lemma.
	\begin{lemma} \label{lem: mu bubble reduction}($\mu$-bubble Reduction Lemma In Cubical Region)
		Let $X, h, d_i$ be as above. Assume that $Sc^{\rtimes}_N(X)\ge\sigma>C_{n,k,N}(d_i)$, $n+k+N\le 7$, where
		\begin{align}\label{eq: Cnkdi}
			C_{n,k,N}(d_i) = \frac{4(n+N+k-1)\pi^2}{n+N+k}\cdot\sum_{i=1}^{n}\frac{1}{d_i^2}.
		\end{align}
		Assume further that $h$ is nontrivial in homology. Then there exists a smooth embedding surface $\Sigma^k$ in $X$ representing the homology class $h$, such that
		\begin{align*}
			Sc_{n+N}^{\rtimes}(g_{\Sigma})\ge \sigma-C_{n,k,N}(d_i)
		\end{align*}
		Here $g_{\Sigma}$ is the induced metric on $\Sigma$ from $X$.
	\end{lemma}
	
	For the proof, we need the following {\it Equivarent Seperation Lemma} proposed by Gromov, see Section 5.4 in \cite{Gro23}, which was proved by $\mu$-bubble. A detailed proof of this lemma could also be found in \cite{WY23}.
	
	\begin{lemma}\label{lem: equivariant seperation}(Equivariant Seperation Lemma)
		Let $X$ be a $m$-dimensional Riemannian band $m\le 7$, possibly non-compact or non-complete. $d = \width(X)$, $Sc(X)>\sigma$. 
		
		Then there is a smooth hypersurface $Y$ seperating $\partial_-X$ and $\partial_+X$, such that
		\begin{align*}
			Sc_1^{\rtimes}(Y)>\sigma-\frac{4(m-1)\pi^2}{md^2}
		\end{align*}
		
		Moreover, if $X$ admits an isometric action by a compact Lie group $G$, then so is $Y$ and the function $\phi$ on $Y$ used to define the $\mathbb{T}^1$-stabilized scalar curvature.
	\end{lemma}
	
	\begin{proof}[Proof of Lemma \ref{lem: mu bubble reduction}]
		Fixing $k,N$, we will make induction on $n$. when $n=1$, the conclusion is just what Lemma \ref{lem: equivariant seperation} says. Assume the conslusion is true for $n-1$, let us consider the case for $n$. Denote $\partial' = \bigcup_{i=2}^n \partial_{\pm i}$ and $\partial'X = \bigcup_{i=2}^n \partial_{\pm i} X$. By a free boundary version of Lemma \ref{lem: equivariant seperation}, there exists a hypersurface $(Y,\partial Y)\subset (X,\partial'X)$, seperating $\partial_{\pm 1}X$, with 
		\begin{align*}
			Sc_{N+1}^{\rtimes}(Y)\ge \sigma-\frac{4(n+N+k-1)\pi^2}{n+N+k}\cdot\frac{1}{d_1^2}.
		\end{align*}
		To see how the $\mathbb{T}$-stabilized scalar curvature plays its role, one just need to apply Lemma \ref{lem: equivariant seperation} to the stabilized space $\mathbb{T}^N\rtimes X$. Since $Y$ is obtained by $\mu$-bubble, it is clear that $Y$ and $\partial_{-1}X$ bound a region in $(X,\partial'X)$. By rewriting $[-1,1]^n = [-1,1]\times I^{n-1}$ and recall $\partial_{-1}X = f^{-1}(\lbrace -1\rbrace\times I^{n-1})$, we have
		\begin{align*}
			[Y] = f_!([\lbrace 0\rbrace\times I^{n-1}])\in H_{n+k-1}(X,\partial'X)
		\end{align*}
		Here we regard $[Y]$ as homology class in $H_{n+k-1}(X,\partial'X)$ and $[\lbrace 0\rbrace\times I^{n-1}]$ as homology class in $H_{n-1}([-1,1]^n,\partial')$.
		
		Consider $f_1 = \pi\circ f: Y\longrightarrow I^{n-1}$, where $\pi: [-1,1]\times I^{n-1}\longrightarrow I^{n-1}$ denotes the projection map. Such $f_1$ is compatible with the cube structure of $Y$: $\partial_{\pm i}Y = \partial_{\pm i}X\cap Y$ for $i=2,3,\dots,n$, while $\partial_{side}Y = \partial_{side}X\cap \partial Y$. 
		
		We claim $f_!(0)\ne 0\in H_k(Y,\partial_{side}Y)$. Without loss of generality, assume $0\in I^{n-1}$ is a regular value of $f_1$. Then the element $f_{1!}(0)$ is represented by $Y\cap f^{-1}(\pi^{-1}(0)) = Y\cap f^{-1}([-1,1]\times\lbrace 0 \rbrace)$. On the other hand, $Y$ and $f^{-1}(\lbrace 0 \rbrace\times I^{n-1})$ represents the same class in $H_{n+k-1}(X,\partial'X)$, we have $Y\cap f^{-1}([-1,1]\times\lbrace 0 \rbrace)$ and $f^{-1}(\lbrace 0 \rbrace\times I^{n-1})\cap f^{-1}([-1,1]\times\lbrace 0 \rbrace) = f^{-1}(0)$ represent the same class in $H_k(X,\partial_{side}X)$. Since $[f^{-1}(0)] = h\ne 0$, the claim is true. Then by the induction hypothesis in dimension $n-1$, we conclude the proof of the lemma.

	\end{proof}
	
	\section{Topological setting for weakly relative aspherical pair}
	
	In this section, we investigate topological properties for weakly relative aspherical pair, which will be repeatedly used in our proof of the main theorems. For general weakly relative aspherical pair $(Y,X)$, we can always pass $Y$ to its {\it relative universal covering} $\Tilde{Y}$, such that $\pi_1(X)\longrightarrow \pi_1(\Tilde{Y})$ is surjective. This yields $\pi_i(\Tilde{Y},X)=0$ for $1\le i\le k$. Throughout this section for simplicity of the notation we still use $Y$ to represent this {\it relative universal covering} $\Tilde{Y}$.
	
	\begin{lemma}\label{lem: intersection}
		Let $Y^{n+k}$ be an oriented manifold and $i:X^n\longrightarrow Y^{n+k}$ an embedded closed submanifold with trivial normal bundle, such that $\pi_i(Y,X)=0, i=1,2,\dots,k$. Let $V$ be the tubular neighbourhood of $X$ in $Y$. Then
		
		(1) $H_i(Y,X)=0, i = 1,2,\dots,k$.
		
		(2) There holds the isomorphism
		\begin{align}\label{eq: topo isomorphism 1}
			H_{k-1}(Y\backslash X) = H_{k-1}(\mathbb{S}^{k-1}\times X) = \mathbb{Z}\oplus H_{k-1}(X)
		\end{align}
	\end{lemma}

	\begin{proof}
		(1) follows directly from the Hurewicz's Theorem. Since $X$ has trivial normal bundle in $Y$, the small tubular neighbourhood of $X$ in $Y$ is diffeomorphic to $V=X\times D^k$. Then by using the excision lemma we have that
		\begin{align*}
			H_i(Y,X) = H_i(Y,V) = H_i(Y-V,\partial V) = 0, i=1,2,\dots,k
		\end{align*}
		Thus we conclude 
		\begin{align}\label{1}
			H_{k-1}(\partial V)\longrightarrow H_{k-1}(Y-V)
		\end{align}
		is an isomorphism. (2) then follows from the fact that $X$ is a deformation retract of $V$ and the  K{\"{u}}nneth formula.		
	\end{proof}
	
	As a result of Lemma \ref{lem: intersection}, we can define the homomorphisms
	\begin{align*}
		&\zeta: H_{k-1}(Y\backslash X)\longrightarrow \mathbb{Z}\\
		&\eta: H_{k-1}(Y\backslash X)\longrightarrow H_{k-1}(X)
	\end{align*}
	by composing the \eqref{eq: topo isomorphism 1} and the projection to the $\mathbb{Z}$ summand and the projection to $H_{k-1}(X)$ respectively. 
 
        The next Lemma investigates the relationship between winding number and $\zeta$-image.
        \begin{lemma}\label{lem: intersection 1}
            Let $X,Y$ be as in Lemma \ref{lem: intersection}. Let $\Gamma^{k-1}$ be an oriented submanifold in $Y\backslash X$ which is the boundary of a $k$-dimensional oriented submanifold in $Y$. Then the winding number of $\Gamma$ and $X$ is well defined, and equals $\zeta([\Gamma])$
        \end{lemma}
        \begin{proof}
            Let $\Gamma = \partial\Sigma$. Recall in usual sense, we always define the winding number of $\Gamma$ and $X$ as the oriented intersection number of $\Sigma$ and $X$. Without of loss of generality we assume $\Sigma$ and $\partial V$ intersect transversally. Denote $\Gamma' = \Sigma\cap\partial V$. Let $\Sigma_0$ be the portion of $\Sigma$ bounded by $\Gamma'$ in $V$. We have $[\Gamma]=[\Gamma']$ in $H_{k-1}(Y\backslash X)$ since $\partial(\Sigma\backslash\Sigma_0) = \Gamma-\Gamma'$. 
            
            Let $\alpha_S\in H^{k-1}(S^{k-1})$ and $\alpha_D\in H^k(\mathbb{D}^k,S^{k-1})$ be the fundamental class for cohomology group, it is clear that $\delta\alpha_S = \alpha_D$. Let $\pi: (V,\partial V) = (\mathbb{D}^k\times X,S^{k-1}\times X)\longrightarrow (\mathbb{D}^k,S^{k-1})$ be the projection. It is not hard to see $\pi^*\alpha_D\in H^k(V,\partial V)$ is the Poincare dual of $i_*[X]\in H_n(V)$. This enables us to compute using Lemma \ref{lem: counting}:
            \begin{align*}
                &[\Sigma]\cdot i_*[X] = [\Sigma_0]\cdot i_*[X] = D_V(i_*[X])([\Sigma_0]) = \pi^*\alpha_D([\Sigma_0])\\
                = &\delta \pi^*\alpha_S([\Sigma_0]) = \pi^*\delta\alpha_S([\Sigma_0])
                = \pi^*\alpha_S(\partial[\Sigma_0]) = \pi^*\alpha_S([\Gamma'])\\
                = &\alpha_S([\pi_*\Gamma']) = \zeta([\Gamma']) = \zeta([\Gamma])
            \end{align*}
            Here $\partial$ denotes the boundary homomorphism for relative homology. Since the result is independent of the choice of $\Sigma$, the winding number is well defined, and this completes the proof of the lemma.
        \end{proof}

        \begin{remark}\label{rem: intersection noncompact}
            If $X$ is noncompact and other conditions are the same, then the conclusion of Lemma \ref{lem: intersection 1} still holds true by taking large region $\Omega$ on $X$ and compute by using relative homology $H_*(\Omega,\partial\Omega)$.
        \end{remark}

        \begin{lemma}\label{lem: intersection 2}
            For an oriented closed manifold $\Sigma^k$, the oriented intersection number of $\Sigma^k$ and $X^n$ equals $0$
        \end{lemma}
        \begin{proof}
            Pick a small sphere $\Gamma^{k-1}=S^{k-1}$ away from $X$, which divides $\Sigma$ into two parts, $\Sigma_1$ and $\Sigma_2$, where $\Sigma_1$ is a topological $\mathbb{D}^k$ away from $X$. By Lemma \ref{lem: intersection 1},
            \begin{align*}
                [\Sigma_1]\cdot[X]=[\Sigma_2]\cdot[X]=0
            \end{align*}
            This completes the proof.
        \end{proof}
        
 The following facts are easy to see, we collect them as lemmas for later use.
	
	\begin{lemma}\label{lem: topology set up}
		\quad
  
		(1) A $(k-1)$-chain $\sigma$ supported in $Y\backslash X$ is zero-homologous in $Y$ if and only if $\eta([\sigma]) = 0$.
		
		(2) A $(k-1)$-chain $\sigma$ supported in $Y\backslash X$ which is zero-homologous in $Y$ is zero-homologous in $Y\backslash X$ if and only if $\zeta([\sigma])=0$.
		
	\end{lemma}
        \begin{proof}
            By excision Lemma, it suffice to deal with the case that $\sigma$ supports on $\partial V = S^{k-1}\times X$. Then the conclusion follows from definition.
        \end{proof}

        \begin{lemma}\label{lem: noncompact}
            If the image of $i_*: H_n(X)\longrightarrow H_n(Y)$ is infinite cyclic, then $Y$ must be noncompact.
        \end{lemma}
        \begin{proof}
            If not, then one can always find a $k$-dimensional oriented submanifold $Z$ with non-zero intersection number with $X$. This is a contradiction with Lemma \ref{lem: intersection 2}.
        \end{proof}

        We remark that in order to guarantee the noncompactness of $Y$, the condition in Lemma \ref{lem: noncompact} could not be removed. In fact, $(S^5,S^3)$ is a weakly relative aspherical pair with isomorphic fundamental group. However, $S^5$ is compact.

	\section{Spherical width and proof of Theorem \ref{thm1}}
	In this section, we focus on the proof of Theorem \ref{thm1}. We first discuss in subsection 4.1 the definition of width of certain homology class, and obtain a lower bound estimate when the homology class runs to infinity by quantitative topology. Next in subsection 4.2, under PSC assumption, we shall establish an upper bound estimate for spherical width in enlargeable setting. These two aspects would finally lead to the desired contradiction.
	\subsection{Width of homology class and its estimate at infinity}
	We first carry out the definiton in the most general setting:
	\begin{definition}\label{Defn: homological width}
		Let $M$ be a Riemannian manifold and $U$ an open subset of $M$. Let $a$ be a homology class in $M$. Define the homological width of $a$ respect to $U$ to be
		\begin{align*}
			W_U(a) = \inf \lbrace\diam\sigma: \sigma \mbox{ is a chain in } M,[\sigma] = a, \supp(\sigma)\subset U\rbrace
		\end{align*}
	\end{definition}
	If for any exhaustion $K_1\subset K_2\subset\dots$ of $M$, $\liminf_{i\to +\infty}{W_{M\backslash K_i}(a)} =\infty$, then we say $a$ has infinite width at infinity. One could get a feeling in the following example:
	\begin{example}
		(1) Consider $(\mathbb{R}^n,g_{Euc})$ with $p\in \mathbb{R}^n$. Let $U = \mathbb{R}^n\backslash p$. Then the generator of $H_{n-1}(\mathbb{R}^n\backslash p)\cong\mathbb{Z}$ has infinite width at infinity.
		
		(2) Consider a metric $g$ on $\mathbb{R}^n$, which is isometric to $(S^{n-1}\times [1,\infty),dt^2+g_{S^{n-1}})$ outside a compact ball. Let $U = \mathbb{R}^n\backslash p$. Then the width of the generator of $H_{n-1}(\mathbb{R}^n\backslash p)\cong\mathbb{Z}$ equals $\diam S^{n-1} = \pi$.
	\end{example}
	
	At the first step, we shall handle in the most general setting, {\it i.e.} the {\it weakly relative aspherical condition}. Suppose $Y^{n+k}$ is weakly relative aspherical to  $X^n$. Fix a metric $g_Y$ on $Y$ and let $X$ inherit the induced metric from $Y$. Let $q:\tilde{Y}\longrightarrow Y$ be the Riemannian covering of $Y$ corresponding to the fundamental group of $\pi_1(X)$. This enables us to lift $X$ to a submanifold $X_0$ of $\tilde{Y}$. Since $Y$ is weakly aspherical relative to $X$, we have that
	\begin{align*}
		&\pi_i(X)\longrightarrow\pi_i(Y) \mbox{ is an isomorphism for } i = 2,3,\dots,k-1\\
            &\pi_k(X)\longrightarrow\pi_k(Y) \mbox{ is surjective }
	\end{align*}
	Hence we get
		\begin{equation}\label{eq: pi-isomorphism}
		    \begin{split}
		        &\pi_i(X_0)\longrightarrow\pi_i(\tilde{Y}) \mbox{ is an isomorphism for } i = 1,2,3,\dots,k-1\\
                    &\pi_k(X_0)\longrightarrow\pi_k(\tilde{Y}) \mbox{ is surjective }
		    \end{split}
		\end{equation}

	This shows $\pi_i(\Tilde{Y},X_0) = 0$ for $i\le k$. By the Hurewicz's Theorem, $H_i(\Tilde{Y},X_0) = 0$ for $i\le k$. Therefore
	\begin{align}\label{eq: H-isomorphism}
		H_i(X_0)\longrightarrow H_i(\Tilde{Y}) \mbox{ is an isomorphism for } i = 1,2,3,\dots,k-1
	\end{align}
	
	The pullback of $X$ under the covering map $q$ is the union of copies of $X$:
	\begin{align*}
		&q^{-1}(X) = \bigcup_{\alpha\in G} X_{\alpha}.\\
		&G = \pi_1(Y)/\pi_1(X)
	\end{align*}
	
	\begin{lemma}\label{lem: L_1}
		There is a universal constant $L_1$ relying only on $(Y,g_Y)$, such that for any $y\in \Tilde{Y}$, there exists $\alpha\in G$, satisfying $\dist(y,X_{\alpha})\le L_1$.
	\end{lemma}
	\begin{proof}
		Since $Y$ is compact, we can pick a point $x\in X$ such that $\dist(x,q(y))< \diam(Y)+1$. Let $\gamma$ be the path in $\Tilde{Y}$ connecting $x$ and $q(y)$. Lift $\gamma$ to a path $\Tilde{\gamma}$ in $\Tilde{Y}$ with endpoints $x'$ and $y$. Then $x$ lies in $X_{\alpha}$ for some $\alpha$. We have $\dist(y,X_{\alpha})\le \diam(Y)$, and this completes the proof of Lemma \ref{lem: L_1}.
	\end{proof}
	
	By \eqref{eq: pi-isomorphism}, $(\Tilde{Y},X_0)$ clearly satisfies the assumption of Lemma \ref{lem: intersection}, that is to say, we have the isomorphism:
	\begin{align}\label{eq: topo isomorphism}
		H_{k-1}(\Tilde{Y}\backslash X) \cong \mathbb{Z}\oplus H_{k-1}(X)
	\end{align}
	Also, we have the maps
	\begin{align*}
		&\zeta: H_{k-1}(\Tilde{Y}\backslash X_0)\longrightarrow \mathbb{Z}\\
		&\eta: H_{k-1}(\Tilde{Y}\backslash X_0)\longrightarrow H_{k-1}(X)
	\end{align*}

	We could now prove the proposition on lower bound estimate for certain homology class in $\Tilde{Y}\backslash X_0$:
	\begin{theorem}\label{thm: sphere width estimate}
		There exists a function $f:(0,+\infty)\longrightarrow (0,+\infty)$, $\lim_{r\to+\infty}f(r)=+\infty$, satisfying the following property: If $a\in H_{k-1}(\Tilde{Y}\backslash X_0)$ satisfies $\zeta(a)\ne 0$, with $a$ supported in $\Tilde{Y}\backslash B_r(X_0)$, then $W_{\Tilde{Y}\backslash B_r(X_0)}(a)>f(r)$.
	\end{theorem}
	\begin{proof}

		Assume the theorem is not true, then there is a constant $C$ and $R_i\to\infty$ and $k-1$-chain $a_i$ in $\Tilde{Y}\backslash X_0$ with $\zeta([a_i])\ne 0$, such that
		\begin{equation}\label{eq: 112}
			\begin{split}
				\supp(a_i)\cap B_{R_i}(X_0) = \emptyset\\
				\diam(a_i)<C_0
			\end{split}
		\end{equation}
		By \eqref{eq: topo isomorphism} we assume $[a_i]=\beta_i+\theta_i$, with $\beta_i\in \mathbb{Z}$ and $\theta_i\in H_{k-1}(X)$. It is clear that $\zeta([a_i])\ne 0$ is equivalent to say $\beta_i\ne 0$.
		
		By Lemma \ref{lem: L_1} there is a copy $X_{\alpha}$ of $X$ such that $\dist(X_\alpha, a_i)<L_1$. By \eqref{eq: H-isomorphism}, we can find a chain $c_i$ supported in $X_{\alpha}$ representing the class $\theta_i$. This implies $\eta([a_i-c_i]) = 0$. Therefore the chain $a_i-c_i$ is homologous to zero in $\Tilde{Y}$. We have the diameter estimate
		\begin{align*}
			\diam(a_i-c_i)\le C_0+L_1+\diam(X)
		\end{align*}
		By Lemma \ref{lem: HZ23}, there is a $k$-chain $\tau_i$, such that $\partial\tau_i = a_i-c_i$, and
		\begin{align*}
			\diam(\tau_i)\le C_1 = C(C_0,g_Y)
		\end{align*}
		Since $\zeta([a_i-c_i]) = \zeta(a)\ne 0$, we have $[a_i-c_i]\ne 0\in H_{k-1}(\tilde{Y}\backslash X_0)$ by Lemma \ref{lem: topology set up}. Therefore, the intersection of the support of $\tau_i$ and $X_0$ is nonempty. This shows:
		\begin{align*}
			\dist(a_i-c_i,X_0)\le C_1
		\end{align*}
		On the other hand, by \eqref{eq: 112}, we have the distance estimate
		\begin{align*}
			\dist(a_i-c_i,X_0)\ge R_i-L_1-\diam(X)
		\end{align*}
		A contradiction is obtained by letting $R_i\to\infty$.
	\end{proof}
	
	\subsection{Reduction to sphere width estimate}
	In this subsection, contrary to the last subsection, we study how PSC condition gives upper bound estimate for our homological width. This estimate is closely related to a so called {\it dominated} $S^2$ {\it stability property} (Definition \ref{defn: S^2 stability}) for $X$. We would obtain this estimate for enlargeable manifold. Then we would assume fully relative aspherical condition, and prove a reduction theorem from Conjecture \ref{conj: RASC} to the {\it dominated} $S^2$ {\it stability property} in codimension 3 case. The results in Section 3 and Sec. 4.1 automatically holds since it is obvious that fully relative aspherical condition implies weakly relative aspherical condition.
	
	To illustrate our point clearer, we shall focus on the case of 2-dimensional homology class. The general case for homological width would sometimes be similar. By using the cube inequality Lemma \ref{lem: mu bubble reduction}, we're able to show the following:
	\begin{lemma}\label{lem: 2-sys estimate}
		Let $X^{n-2}$ be an enlargeable manifold and $Y^n$ a compact Riemannian manifold with $Sc_N^{\rtimes}(Y)\ge\sigma$ $(n+N\le 7)$. Assume there is a nonzero degree map $f:Y^n\longrightarrow S^2\times X^{n-2}$. Then there is an embedded 2-sphere $\Sigma$ in $Y^n$, with
		\begin{align}\label{eq: 2}
			\diam \Sigma\le \frac{C(n,N)}{\sqrt{\sigma}}
		\end{align}
		which also satisfies the property that the image of $[\Sigma]$ under the composition of following maps
		\begin{align}\label{eq: 3}
			H_2(\Sigma)\stackrel{f_*}{\longrightarrow}H_2(S^2\times X)=H_2(S^2)\oplus H_2(X)\longrightarrow H_2(S^2)
		\end{align}
		does not vanish. Here the last map means projection on the first summand.
	\end{lemma}
	\begin{proof}
		By Lemma \ref{lem: enlargeable}, for any $d_0>0$ there exists a covering $\Tilde{X}$ of $X$ and a cube like region $V$ in $\Tilde{X}$, such that
		\begin{align*}
			\dist(\partial_{-i}V,\partial_{+i}V)>d_0, \mbox{ for } i = 1,2,\dots,n
		\end{align*}
		and a non-zero degree map
		\begin{align*}
			\varphi: V\longrightarrow [-1,1]^{n-2}
		\end{align*}
		Let $\Tilde{Y}$ be the pullback object in the following diagram
		\begin{equation*}
			\xymatrix{&S^2\times\Tilde{X}\ar[r]&S^2\times X\\
				&\Tilde{Y}\ar[r]\ar[u]^{\Tilde{f}}&Y\ar[u]_f}
		\end{equation*}
		Denote $\Omega_0 = S^2\times V$ and let $\Omega = \Tilde{f}^{-1}(\Omega_0)\subset \Tilde{Y}$. Then $\Omega$ is also a cube like region with
		\begin{align*}
			\dist(\partial_{-i}\Omega,\partial_{+i}\Omega)>d = \frac{d_0}{\Lip f}, \mbox{ for } i = 1,2,\dots,n
		\end{align*}
		
		Denote $h$ be the homology class in $H_2(\Omega)$ obtained by pulling back $0$ by $\varphi\circ q\circ \Tilde{f}$, where $q: S^2\times \tilde{X}\longrightarrow \tilde{X}$ is the projection map. Denote $\alpha = D_{\Omega_0}([S^2])\in H^n(\Omega_0,\partial\Omega_0)$ to be the Poincare dual of $[S^2]\in H_2(\Omega_0)$. Also denote $\beta = [S^2]^*\in H^2(\Omega_0)$ to be the canonical cohomology class with evaluation $1$ on $[S^2]$, which obviously exists since $[S^2]$ is free. 
		\begin{equation}\label{eq: 1}
			\begin{split}
				h &= \tilde{f}_!(q_!\varphi_!(0)) = \degg\varphi\cdot\tilde{f}_!(D_{\Omega_0}(\alpha))\\
				&= \degg\varphi\cdot D_{\Omega}(\tilde{f}^*(\alpha))
			\end{split}
		\end{equation}

		It is clear that $\Tilde{f}$ has non-zero degree restricted on $\Omega$, which shows
		\begin{align*}
			0\ne \Tilde{f}^*([\Omega_0]^*) = \Tilde{f}^*(\alpha\smallsmile\beta) =  \Tilde{f}^*(\alpha)\smallsmile\Tilde{f}^*(\beta)
		\end{align*}
		Combining with \eqref{eq: 1}, $h\ne 0$.
		
		Then we can apply Lemma \ref{lem: mu bubble reduction} to find a closed surface $\Sigma_{pre}$ in $\Omega$ representing the class $h$, with
		\begin{align*}
			Sc_{n+N-2}^{\rtimes}(\Sigma_{pre})\ge &\sigma-\frac{4(n+N-1)\pi^2}{n+N}\cdot\sum_{i=1}^{n-2}\frac{1}{d_i^2}\\
			>&\sigma - \frac{4(n+N-1)\pi^2}{n+N}\cdot \frac{n-2}{d^2}
		\end{align*}
		Choose $d_0 = (\Lip f)d$ to be sufficiently large, we have $Sc_{n+N-2}^{\rtimes}(\Sigma_{pre})>\frac{\sigma}{2}$. Then by \cite{Gro20} ( Page 2, Example 1 ) we get the desired diameter bound for each component of $\Sigma_{pre}$.
		
		Since $\degg\Tilde{f}\ne 0$, it's not hard to see
		\begin{align*}
			\tilde{f}_*(\Sigma_{pre}) = \tilde{f}_*(h) = \tilde{f}_*(\tilde{f}_!((\degg\varphi) [S^2])) = (\degg\tilde{f})(\degg\varphi) [S^2]
		\end{align*}
		For the last equality we have used
		\begin{align*}
			\tilde{f}_*(\tilde{f}_!(a)) &= \tilde{f}_*(D_{\Omega}(\tilde{f}^*D_{\Omega_0}a)) = \tilde{f}_*([\Omega]\smallfrown \tilde{f}^*D_{\Omega_0}a)\\
			&= \tilde{f}_*([\Omega])\smallfrown D_{\Omega_0}a = (\degg\tilde{f})a.
		\end{align*}
		It follows that there exists a component of $\Sigma_{pre}$, whose image under $\Tilde{f}_*$ has non-zero part over $[S^2]$. Then the image of this component under the covering map $\Tilde{Y}\longrightarrow Y$, denoted by $\Sigma$, has the desired property.

	\end{proof}
	\begin{remark}
		Different from 2-systole estimate \cite{Zhu20}, in which one needed only to guarantee the small sphere found to be homotopically nontrivial, in our case we have to carefully record the homological information of the small 2-sphere for later use.
	\end{remark}

	Now let us make the following definition:
	\begin{definition}\label{defn: S^2 stability}
		Let $X^{n-2}$ be a differential manifold. We say $X^{n-2}$ has dominated $S^2$-stability property, if for any Riemannian manifold $Y^n$ which admits a non-zero degree map $f: Y\longrightarrow S^2\times X^{n-2}$, satisfying $Sc_1^{\rtimes}(Y)\ge\sigma$, one could always find a embedded sphere $\Sigma$, satisfying \eqref{eq: 2}\eqref{eq: 3}. 
	\end{definition}
	
	It follows from Lemma \ref{lem: 2-sys estimate} that enlargeable manifold of dimension $n$ ($n+3\le 7$) has the dominated $S^2$-stability property. The proof of Theorem \ref{thm1} in codimension 3 then follows from the following reduction proposition:
	\begin{proposition}\label{prop: reduction}
		If $X^n$ has the dominated $S^2$-stability property, then Conjecture \ref{conj: RASC} holds true for $(Y^{n+3},X^n)$ provided $n+3\le 7$.
	\end{proposition}

	\begin{proof}
		Assume the conclusion is not true, by compactness there is a metric $g_Y$ on $Y$ such that $Sc(g_Y)>2$. We pass $Y$ to its covering $\Tilde{Y}$ as in the preceding section. Since $Y$ is fully relative aspherical to $X$, the inclusion of $X_0$ to $\Tilde{Y}$ induces isomorphism on homotopy groups in all dimensions. Hence, by the Whitehead's Theorem there is a deformation retraction map $\pi: \Tilde{Y}\longrightarrow X_0$.
		
		Denote $U_{\epsilon}\cong \mathbb{D}^3\times X_0$ to be the small tubular neighbourhood of $X_0$ in $\Tilde{Y}$, with $\partial U_{\epsilon} = S^2\times X_0$. Let $p: S^2\times X_0\longrightarrow X_0$ be the projection, we claim there is a map $\Tilde{\pi}: \Tilde{Y}\backslash U_{\epsilon}\longrightarrow \partial U_{\epsilon}$ such that the following diagram commute
		\begin{equation*}
			\xymatrix{\Tilde{Y}\backslash U_{\epsilon}\ar[r]^{\tilde{\pi}}\ar[rd]_{\pi}& \partial U_{\epsilon}\ar[d]^p\\
				&X_0}
		\end{equation*}
		This follows from the obstruction theory. In fact, all of the obstruction of the lifting lies in the homology group $H^{n+1}(\Tilde{Y}\backslash U_{\epsilon}, \partial U_{\epsilon},\pi_n(S^2))$, which equals to zero since by exision lemma we have
		\begin{align}\label{eq: no obstruction}
			H^{n+1}(\Tilde{Y}\backslash U_{\epsilon}, \partial U_{\epsilon},\pi_n(S^2)) = H^{n+1}(\Tilde{Y}, U_{\epsilon}, \pi_n(S^2))=
			H^{n+1}(\Tilde{Y}, X_0, \pi_n(S^2)) = 0
		\end{align}
		Such operation also guarantees $\Tilde{\pi}|_{\partial U_{\epsilon}}=id_{\partial U_{\epsilon}}$.
		
		In $\Tilde{Y}\backslash U_{\epsilon}$ we are able to construct a Riemannian band $\mathcal{V} = \lbrace x\in \Tilde{Y}\backslash U_{\epsilon}, R\le\dist(x,\partial U_{\epsilon})\le R+L$. Let $L$ be sufficiently large, then by the standard $\mu$-bubble argument there is a hypersurface $\Sigma$ seperating two ends of $\mathcal{V}$ such that
		\begin{align*}
			Sc_1^{\rtimes}(\Sigma)\ge Sc_Y-\frac{4(n+2)}{n+3}\frac{\pi^2}{L^2}\ge 2-\frac{4(n+2)}{n+3}\frac{\pi^2}{L^2}\ge 1
		\end{align*}
		
		Define $f = \Tilde{\pi}|_{\partial U_{\epsilon}}$, we have $\degg f = \degg \Tilde{\pi}|_{\partial U_{\epsilon}} = 1$. By our assumption (the case that $X$ is enlargeable follows from Lemma \ref{lem: 2-sys estimate}) there is a 2 sphere $\Gamma$ in $\Sigma$, such that the image of $[\Gamma]$ under the composition of the following maps does not vanish:
		\begin{align*}
			H_2(\Gamma)\stackrel{f_*}{\longrightarrow}H_2(S^2\times X)=H_2(S^2)\oplus H_2(X)\longrightarrow H_2(S^2)
		\end{align*}
		Also, we have the diameter bound
		\begin{align}\label{eq: 78}
			\diam(\Gamma)<C
		\end{align}
		By using the notation in the previous subsection, this is equivalent saying
		\begin{align}\label{eq: 79}
			\zeta([\Gamma])\ne 0
		\end{align}
		
		On the other hand, $R$ could be taken arbitrarily large when constructing $\mathcal{V}$. This shows $\Sigma$ can be sufficiently far away from $X_0$. An contradiction then follows from Theorem \ref{thm: sphere width estimate}, \eqref{eq: 78}, \eqref{eq: 79}. 
	\end{proof}
	
	\begin{proof}[Proof of Theorem \ref{thm1}]
		If $k=3$, then the conclusion follows from Proposition \ref{prop: reduction}. If $k=4$, by similar argument, for any $R>0$ there is a 3-dimensional embedded submanifold $\Gamma$ in $\Tilde{Y}$, such that
		\begin{equation}\label{eq: 85}
			\begin{split}
				&\dist(\Gamma,X_0)>R\\
				&Sc_{n+1}^{\rtimes}(\Gamma)\ge 1\\
				&\zeta([\Gamma])\ne 0
			\end{split}
		\end{equation}
		By the refined slice and dice Lemma \ref{lem: slice and dice}, there are disjoint collection of spheres $S_{\alpha}$ dividing $\Gamma$ into regions $U_i$, with $\diam(S_{\alpha})<L_0$, $\diam(U_i)<L_0$, and
            \begin{align*}
                \partial U_i = \bigcup_{\alpha\in A(i)} S_{\alpha}
            \end{align*}
  Since the Hurewicz map $\pi_2(X)\longrightarrow H_2(X)$ vanishes, by \eqref{eq: H-isomorphism} we have the Hurewicz map $\pi_2(\Tilde{Y})\longrightarrow H_2(\Tilde{Y})$ vanish. Therefore each $S_{\alpha}$ is homologous to zero in $\Tilde{Y}$. By Lemma \ref{lem: HZ23}, there exists 3-chains $T_{\alpha}$, satisfying $\partial T_{\alpha} = S_{\alpha}$ and $\diam T_{\alpha}<L_3 = L_3(L_0,g_Y)$. Denote
		\begin{align*}
			\hat{U}_i = U_i + \sum_{\alpha\in A(i)} (\pm T_{\alpha})
		\end{align*}
		The sign is chosen with respect to the orientation of $S_{\alpha}$ in $U_i$. We have $\partial \hat{U}_i = 0$. Since for each $S_{\alpha}$ we have filled in a pair of $T_{\alpha}$ with opposite orientation, at the level of homology class we have
		\begin{align*}
			[\Gamma] = \sum_{i=1}^u [\hat{U}_i]
		\end{align*}
		Since $\zeta([\Gamma])\ne 0$, there exists $\hat{U}_i$ such that $\zeta(\hat{U}_i)\ne 0$. By \eqref{eq: 85}, we have $\dist(\hat{U}_i,X_0)>R-L_3$, $\diam\hat{U}_i<L_0+L_3$. A contradiction then follows from Theorem \ref{thm: sphere width estimate} by letting $R\to +\infty$. This finishes the proof of Theorem \ref{thm1}.
	\end{proof}
	
	In light of the reduction Proposition \ref{prop: reduction}, for further understanding of Conjecture \ref{conj: RASC} for more general class of manifold $X$, it remains an important but maybe not easy problem of discussing whether the manifold admitting no PSC metric has certain $S^2$-stability property like Definition \ref{defn: S^2 stability}. Note here we could no more require the dominated $S^2$-stability property any more, since there does exists the example that $X$ admits no PSC metric, but manifold with degree 1 to $X$ admits PSC metric. To overcome this difficulty, it seems that one may need to appeal to generalized surgery argument developed in \cite{CRZ23}\cite{R23}.
	
	Though in general $S^2$-stability may be hard, at least we have the following slight extension of the class of such kind of manifold: Let $P$ be a parallizable $5$-dimensional closed manifold admitting a metric of non-positive sectional curvature, and $X^4$ a submanifold representing a non-zero homology class in $H_4(P,\mathbb{Q})$. Then $X^4$ has the dominated $S^2$-stability property. This could be proved by directly applying the method of the second proof of Theorem 13.8 in \cite{GL83}, combining with the argument used in the proof of Lemma \ref{lem: 2-sys estimate}. As a result, Conjecture \ref{conj: RASC} holds true for such $X$ in codimension 3.
	
	At the end of this subsection, we hope to raise the following conjecture, which seems not too farfetched by minimal hypersurface method.
	
	\begin{conj}
		Let $X^n$ be a Schoen-Yau-Schick manifold $(n\le 5)$, then $X^n$ has degree 1 version of dominated $S^2$-stability property. Here we call a oriented manifold $X^n$ is Schoen-Yau-Schick (see \cite{SY79b}\cite{Sch98}), if there exists $\alpha_1,\alpha_2,\dots,\alpha_{n-2}\in H^1(X)$, such that
		\begin{align*}
			[X]\smallfrown\alpha_1\smallfrown\alpha_2\smallfrown\dots\smallfrown\alpha_{n-2}
		\end{align*}
		does not lie in the Hurewicz image of $\pi_2(X)$.
	\end{conj}

	\section{The weakly relative aspherical condition}
	
	In this section, we prove Theorem \ref{thm2}. Throughout this section $Y$ would be weakly aspherical relative to $X$.
	
	Assume $(Y,g_Y)$ has scalar curvature greater than $2$. Pass $Y$ to its Riemannian covering $\Tilde{Y}$ as what has been done in Sec. 4.1. Let $\phi: X\longrightarrow Z$ be the non-zero degree map to some enlargeable, aspherical manifold $Z$. Since $i:X_0\longrightarrow \Tilde{Y}$ induces isomorphism in $\pi_1$, By Theorem 1B.9 in \cite{Hat02}, there is a map $\Phi:\Tilde{Y}\longrightarrow Z$, such that $\Phi_* = \phi_*\circ (i_*)^{-1}:\pi_1(\Tilde{Y})\longrightarrow \pi_1(Z)$. Moreover, it is clear from our construction that $\Phi\circ i$ and $\phi$ induces the same homomorphism from $\pi_1(\Tilde{Y})$ to $Z$. By the uniqueness part of Theorem 1B.9 in \cite{Hat02}, $\Phi\circ i$ must be homotopic to $\phi$. Therefore, the following diagram commutes up to homotopy:
	\begin{equation*}
		\xymatrix{\Tilde{Y}\ar[rd]_{\Phi}& X_0\ar[d]^{\phi}\ar[l]_i\\
			&Z}
	\end{equation*}
	Without loss of generality we set $\phi = \Phi\circ i$. By small perturbation we can assume the map $\Phi:\tilde{Y}\longrightarrow X_0$ is smooth. Since $\degg\phi\ne 0$, the image of $i_*$ must be infinite cyclic. Hence, $\Tilde{Y}$ must be noncompact owing to Lemma \ref{lem: noncompact}.
	
	We denote $U_R = \lbrace y\in\Tilde{Y},\dist(y,X_0)<R\rbrace$. It is clear that $U_R$ is a compact region in $\Tilde{Y}$ for all $R>0$. Let $d_0$ be sufficiently large such that
	\begin{align}\label{eq: d_0}
		\frac{4(n+k-1)\pi^2}{(n+k)d_0^2}<\frac{1}{2}
	\end{align}
	Note that the region $\mathcal{U}=U_{R+d_0}\backslash U_R$ is a Riemannian band with $\partial U_{R+d_0}\backslash U_R = \partial U_{R+d_0}\cup\partial U_R$, such that
	\begin{align*}
		\mbox{width}(\mathcal{U}) = \dist(\partial U_{R+d_0},\partial U_R)>d_0
	\end{align*}
	in the sense of \cite{Gro18}.
	By the compactness of $U_{R+d_0}$, the differential of $\Phi$ must be bounded from above:
	\begin{align*}
		\Lip\Phi<C \mbox{ on } U_{R+d_0}
	\end{align*}
	
	Next we have to find suitable coverings $(\hat{Y}, \hat{Z})$ of $(\Tilde{Y},Z)$. Let $d$ be sufficiently large such that
	\begin{align}\label{eq: d}
		\frac{4(n+k-1)\pi^2}{(n+k)}\cdot\frac{kC^2}{d^2}<\frac{1}{2}
	\end{align}
	By Lemma \ref{lem: enlargeable}, we can find a covering $\hat{Z}$ of $Z$ and a cube like region $V$ in $\Tilde{Z}$, such that
	\begin{equation}\label{eq: V-dist}
		\begin{split}
			\dist(\partial_{-i}V,\partial_{+i}V)>d, \mbox{ for } i = 1,2,\dots,n\\
			\varphi: V\longrightarrow [-1,1]^n \mbox{ has nonzero degree }
		\end{split}
	\end{equation}
	Let $\hat{Y}\longrightarrow \Tilde{Y}$ be the covering which corresponds the pullback object in the right part of the following diagram, and $\hat{X}$ simply be $p_{\hat{Y}}^{-1}(X_0)$. 
	
	\begin{equation}\label{diagram: 1}
		\xymatrix{\hat{X}\ar[r]^{\hat{i}}\ar[d]^{p_{\hat{X}}}&\hat{Y}\ar[r]^{\hat{\Phi}}\ar[d]^{p_{\hat{Y}}}&\hat{Z}\ar[d]^{p_{\hat{Z}}}\\
			X_0\ar[r]^i&\tilde{Y}\ar[r]^{\Phi}&Z}
	\end{equation}
        Let $\hat{\phi} = \hat{\Phi}\circ\hat{i}$, it is clear that $\degg\hat{\phi} = \degg\phi\ne 0$ (Here the degree is defined for proper map). Let
	\begin{align*}
		\hat{U}_R = &p_{\hat{Y}}^{-1}(U_R), \hat{U}_{R+d_0} = p_{\hat{Y}}^{-1}(U_{R+d_0})\\
		&\Omega = \hat{\Phi}^{-1}(V)\cap\hat{U}_{R+d_0}\\
		\Omega_0 &= \hat{\Phi}^{-1}(V)\cap\hat{U}_{R+d_0}\backslash\hat{U}_R
	\end{align*}
	Therefore $\Omega$ is a cubical region in the sense of Section 2. In fact, we can define
	\begin{align*}
		\partial_{\pm i}\Omega = \hat{\Phi}^{-1}(\partial_{\pm i}V)\\
		\partial_{eff} = \bigcup_{i=1}^n \partial_{\pm i}\Omega\\
		\partial_{side} = \Omega\cap\partial\hat{U}_{R+d_0}\\
	\end{align*}
and $f = \varphi\circ\hat{\Phi}: \Omega\longrightarrow [-1,1]^n$, sending the effective boundary to $[-1,1]^n$. Moreover, we denote
 \begin{align*}
     (W,\partial W) = (\Omega,\partial\Omega)\cap \hat{X}
 \end{align*}
 From our construction, we also have $(W,\partial W) = \hat{\phi}^{-1}(V,\partial V)$. If we denote $\hat{\phi}_0 = \hat{\phi}|_{(W,\partial W)}$, then $\degg \hat{\phi}_0 = \degg\phi\ne 0$.
	
	It follows from the diagram \ref{diagram: 1} that
	\begin{align}\label{eq: Lip}
		\Lip\hat{\Phi}|_{\hat{U}_{R+d_0}} = \Lip\Phi|_{U_{R+d_0}}<C
	\end{align}
	Let $d_i = \dist(\partial_{-i}\Omega,\partial_{+i}\Omega)$, from \eqref{eq: Lip}\eqref{eq: V-dist} we have $d_i>\frac{d}{C}$. Combining with \eqref{eq: d} and recall the definition of $C_{n,k}(d_i)$ in \eqref{eq: Cnkdi}, we have
	\begin{align}\label{eq: Cnkdi 2}
		C_{n,k}(d_i)<\frac{1}{2}
	\end{align}
	
	The next step is to use Lemma \ref{lem: mu bubble reduction} to find small spheres collecting the information of scalar curvature. Let us first examine a non-trivially intersecting condition. Assume $0$ is a regular value of $f$ and let $h\in H_k(\Omega,\partial_{side}\Omega)$ be the homology class representing $f^{-1}(0)$. Let $\hat{i}:\hat{X}\longrightarrow\hat{Y}$ be the inclusion map. It is clear that $\hat{i}_*[W,\partial W$ represents a homology class in $H_n(\Omega,\partial_{eff}\Omega)$. Since $\varphi^{-1}(0)=\lbrace p_1,p_2,\dots,p_l\rbrace\subset V$ with $l = \deg \varphi\ne 0$, we have:
	\begin{align*}
		f^{-1}(0) = \bigcup_{i=1}^l \hat{\Phi}^{-1}(p_i)
	\end{align*}
	By Lemma \ref{lem: counting} we are able to compute:
	\begin{align*}
		&h\cdot \hat{i}_*[W,\partial W] = \sum_{i=1}^l[\hat{\Phi}^{-1}(p_i)]\cdot\hat{i}_*[W,\partial W]\\
		=& \sum_{i=1}^l\hat{\Phi}^*[V,\partial V]^*(\hat{i}_*[W,\partial W])\\
		=& \sum_{i=1}^l[V,\partial V]^*(\hat{\Phi}_*\hat{i}_*[W,\partial W])\\
            =& l\degg\hat{\phi}_0\ne 0
	\end{align*}
	This shows any submanifold representing $h$ has intersection number $l$ with $\hat{X}$. In particular, $h\ne 0$.
	
	By Lemma \ref{lem: mu bubble reduction} and \eqref{eq: Cnkdi 2}, there exists a submanifold $\Sigma^k$ representing $h\in H_k(\Omega,\partial_{side}\Omega)$, such that
	\begin{align*}
		Sc_n^{\rtimes}(\Sigma)\ge Sc(\Omega)-C_{n,k}(d_i)\ge 2-\frac{1}{2} = \frac{3}{2}
	\end{align*}
	Since $\Sigma$ and $\hat{X}$ have non-zero intersection, by Lemma \ref{lem: intersection 2} $\Sigma$ cannot be closed. Therefore $\partial\Sigma\ne\emptyset$ and we have $\partial\Sigma\subset \partial\Omega$. Similarly, one is able to show $\Sigma\cap\partial \hat{U}_{R}\ne\emptyset$, or else the portion of $\Sigma$ in $\hat{U}_R$ is a closed one and one obtains contradiction by Lemma \ref{lem: intersection}.
	
	Denote $\Sigma_{R,R+d_0} = \Sigma\cap (\hat{U}_{R+d_0}\backslash \hat{U}_R)$. It is clear that $\Sigma_{R,R+d_0}$ is a Riemannian band with $\width(\Sigma_{R,R+d_0})>d_0$. By a standard $\mu$-bubble argument as in \cite{CL20}\cite{Gro20}\cite{GZ21}\cite{Zhu23} one could find a submanifold $\Gamma^{k-1}\subset\Sigma_{R,R+d_0}$ which separates $\Sigma\cap\hat{U}_{R+d_0}$ and $\Sigma\cap\hat{U}_{R}$, with
	\begin{align}\label{eq: Sc lower bound}
		Sc_{n+1}^{\rtimes}(\Gamma)\ge Sc_n^{\rtimes}(\Sigma)-\frac{4(n+k-1)\pi^2}{(n+k)d_0^2}\ge\frac{3}{2}-\frac{1}{2}\ge 1
	\end{align}
	Denote the portion of $\Sigma$ bounded by $\Gamma$ to be $\Sigma_0$. Since $\partial\Sigma_0 = \Gamma$ and $\Gamma\cap\hat{X}=\emptyset$, By slight perturbation of $\Sigma_0$ away from $\Gamma$ we can assume $\Sigma_0$ intersects $\hat{X}$ transversally.
	
	We will then carry out our argument back in the original pair $(\Tilde{Y},X_0)$. Since $\Sigma_0$ intersects transversally with $\hat{X} = p_{\hat{Y}}^{-1}(X)$, and note that for $x\in \Sigma_0$,
	\begin{align*}
		p_{\hat{Y}}(x)\in X_0 \mbox{ if and only if } x\in  \hat{X} = p_{\hat{Y}}^{-1}(X)
	\end{align*}
	The geometric intersection number of $p_{\hat{Y}}:\Sigma_0\longrightarrow \Tilde{Y}$ and $X_0$ must equal that of $\Sigma_0$ and $\hat{X}$, hence does not equal to zero. By Lemma \ref{lem: intersection 1}, $\zeta([\Gamma])\ne 0$. As a result, there is a connected component $\Gamma_0$ of $\Gamma$ such that $\zeta([\Gamma_0])\ne 0$.
	
	If $k=3$, by letting $R\to +\infty$, a contradiction follows from diameter estimate and Theorem \ref{thm: sphere width estimate}. If $k=4$, the conclusion follows by exactly the same argument used in the proof of Theorem \ref{thm1}. One only needs to note that the spheres in Lemma \ref{lem: slice and dice} dividing $\Gamma_0$ into parts are homologous to zero by our vanishing Hurewicz map assumption, then we are able to find a 3-chain with bounded diameter, non-zero $\zeta$-image, and supported arbitrarily far away from $X_0$. This contradicts with Theorem \ref{thm: sphere width estimate}.

	\section{Proof of the Corollaries}
	
	In this section, we will prove the corollaries. We begin by recalling the following result by Gromov, which implies characterization for closed 3-manifold admitting no PSC metric.
	\begin{lemma}\label{lem: Gro23}
		(\cite{Gro23}, Chapter 3.10) Let $Z$ be a closed 3-dimensional aspherical manifold, then the universal covering $\Tilde{Z}$ of $Z$ is hyperspherical.
	\end{lemma}
	
	\begin{lemma}\label{lem: 3-enlargeable mfd}
		Let $X$ be a closed 3-dimensional manifold which admits no PSC metric, then $X$ is enlargeable.
	\end{lemma}
	\begin{proof}
		By the classification of PSC 3-manifold \cite{GL83}, each 3-manifold $M$ admitting no PSC metric contains an aspherical factor $X$, which shows that it admits a degree $1$ map to $X$. The result follows immediately from Lemma \ref{lem: Gro23}.
	\end{proof}
	
	\begin{proof}[Proof or Corollary \ref{cor1}]
		The $k=3$ case follows from Lemma \ref{lem: 3-enlargeable mfd} and Theorem \ref{thm1}. The $k=4$ case follows from the fact that for a compact 3-manifold $X$, if it contains no $S^2\times S^1$ factor in its prime decomposition, then the Hurewicz map $\pi_2(X)\longrightarrow H_2(X)$ vanishes. In fact, such manifold is made purely by irreducible factors, each factor has vanishing $H_2$ on their universal covering. By the Hurewicz's Theorem and Mayer-Vietoris Theorem, element representing $\pi_2(X)$ only appears as the connecting sphere used to construct the connected sum, which is obviously homologous to zero in $X$.
	\end{proof}
	
	Next we prove a lemma which will be used in the proof of Corollary \ref{cor2}.
	
	\begin{lemma}
		Let $F$ be an enlargeable manifold, then the $F$ bundle $E$ over $S^1$ is also enlargeable.
	\end{lemma}
	\begin{proof}
		Since enlargeability is topological invariant, we can discuss the problem under fixed metric $g_E$ and $g_F$ on $E$ and $F$. Let $\pi:\mathbb{R}^1\longrightarrow S^1$ be the universal covering and $\Tilde{E} = \pi^*E$, we have the following diagram
		\begin{equation*}
			\xymatrix{&\Tilde{E}\ar[r]^{\Pi}\ar[d]^q &E\ar[d]^{p}\\
				&\mathbb{R}^1\ar[r]^{\pi} &S^1}
		\end{equation*}
		We have
		\begin{align*}
			\Lip q=\Lip p<C_1<+\infty
		\end{align*}
		Note that $\Tilde{E}$ is trivial by the contractibility of $\mathbb{R}^1$. Fix $L>0$ , restrict $\Tilde{E}$ on $[0,L]$ to obtain the bundle $\Tilde{E}_L$. Consider the projection map $r: \Tilde{E}_L\longrightarrow F$. By the compactness of $\Tilde{E}_L$, we have
		\begin{align*}
			\Lip r<C_2<+\infty
		\end{align*}
		We then get the diffeomorphism from $\Tilde{E}_L$ to the Riemannian product $[0,L]\times F$
		\begin{align*}
			\Phi: \Tilde{E}_L \stackrel{q\times r}{\longrightarrow} [0,L]\times F
		\end{align*}
		with
		\begin{align*}
			\Lip\Phi <C_1+C_2
		\end{align*}
		
		For any $\epsilon>0$, we find a covering space $\Tilde{F}$ of $F$ and a map $f:\Tilde{F}\longrightarrow S^n$ with non-zero degree. This induces the map
		$\hat{\Phi}:\hat{E}_L\longrightarrow [0,L]\times \Tilde{F}$, where $\hat{E}_L$ is the covering of $\Tilde{E}_L$ induced by $[0,L]\times \Tilde{F}$. Let $\eta: [0,L]\longrightarrow S^1$ be the composition of the map pinching two ends of $[0,1]$ into a single point and the retraction map $[0,L]\longrightarrow [0,1]$ with $\Lip \eta<\frac{1}{L}$, and $\rho: S^1\times S^n\longrightarrow S^{n+1}$ be a non-zero degree map with $\Lip\rho<C_3=C(n)$. Then,
		\begin{align*}
			F = \rho\circ(\eta\times f)\circ\hat{\Phi} : \hat{E}_L\longrightarrow S^{n+1}
		\end{align*}
		is a map of non-zero degree, with
		\begin{align*}
			\Lip F< C_3(\frac{1}{L}+\epsilon)(C_1+C_2)
		\end{align*}
		By Letting $L\to+\infty$ and $\epsilon\to 0$, $\Lip F$ could be arbitrarily small. And by this construction it is easy to see that $\Tilde{E}$ is also enlargeable.
	\end{proof}
	
	\begin{proof}[Proof of Corollary \ref{cor2}]
		Since $B$ is a closed aspherical manifold, it is well known that one can find a $S^1$ in $B$ such that the homomorphism $\pi_1(S^1)\longrightarrow\pi_1(B)$ induced by the inclusion map is injective. Consider the restricted $F$ bundle of $Y$ on $S^1$, and denote this bundle to be $E$. Since $\pi_2(B) = 0$, by the long exact sequence of the homotopic group of fiber bundles, $F$ is incompressible in $Y$. Consider the following diagram:
		
		\begin{equation*}
			\xymatrix{0\ar[r]&\pi_1(F)\ar[r]\ar[d]^{id}&\pi_1(E)\ar[r]\ar[d]&\pi_1(S^1)\ar[r]\ar[d]&0\\
				0\ar[r]&\pi_1(F)\ar[r]&\pi_1(Y)\ar[r]&\pi_1(B)\ar[r]&0}
		\end{equation*}
		
		By a similar diagram chase as in \cite{He23}, $E$ is also incompressible in $Y$. Since both $S^1$ and $B$ are aspherical, for $i\ge2$ we consider the following diagram:
		
		\begin{equation*}
			\xymatrix{0\ar[r]&\pi_i(F)\ar[r]^{\cong}\ar[d]^{id}&\pi_i(E)\ar[r]\ar[d]&0\\
				0\ar[r]&\pi_i(F)\ar[r]^{\cong}&\pi_i(Y)\ar[r]&0}
		\end{equation*}
		
		Therefore the map $\pi_i(E)\longrightarrow\pi_i(Y)$ is an isomorphism for $i\ge2$, which shows $Y$ is aspherical relative to $E$. Note that the codimension of $E$ in $Y$ is $k-1$.
		
		(1) $k-1 = 3$. The conclusion follows easily from Theorem \ref{thm1} and Lemma \ref{lem: 3-enlargeable mfd}.
		
		(2) $k-1 = 4$. We have $n-k\le 2$ in this case, which simply shows $F$ could only be $S^1$ or closed surface with positive genus, and therefore $\pi_2(F) = 0$.  The conclusion follows from Theorem \ref{thm1} of the same reason.
	\end{proof}
	
	\begin{proof}[Proof of Corollary \ref{cor3}]
		Assume that $Y^n$ deformes to $X^{n-2}$. We say $X$ has dominated twisted $S^1$ stability, if any compact manifold which admits a degree map to any $S^1$ bundle over $X$ admits no PSC metric. When $n-2\le 4$, since $X$ is aspherical, the $S^1$ bundle over $X$ is also aspherical. Then by \cite{CL20}\cite{CLL23}\cite{Gro20}, $X$ has dominated twisted $S^1$ stability, and the result follows from Proposition 5.2 in \cite{He23}. If $n-2=5$, denote $U_{\epsilon}$ to be the tubular neighbourhood of $X$. by Corollary \ref{cor2}, $E=\partial U_{\epsilon}$, the $S^1$ bundle over $X$ admits no PSC metric. It is not hard to verify that $E$ is incompressible in $Y\backslash X$. The result then follows from the generalized surgery argument as in \cite{CRZ23} and a standard $\mu$-bubble argument.
	\end{proof}
	
	\begin{proof}[Proof of Corollary \ref{cor4}]
		If $X$ has trivial normal bundle in $Y$, then the result follows directly from Theorem \ref{thm2}. In general case, we have to do some necessary modification for the proof of Theorem \ref{thm2}. Pass $X$ to its universal covering $\hat{X}$ and we have the following diagram
		\begin{equation}\label{diagram: 2}
			\xymatrix{\hat{Y}\ar[r]^{\hat{\Phi}}\ar[d]^{p_{\hat{Y}}}&\hat{X}\ar[d]^{p_{\hat{X}}}\\
				\tilde{Y}\ar[r]^{\Phi}&X}
		\end{equation}
		Since $X$ is aspherical, $\hat{X}$ is contractible, which yields the normal bundle of $\hat{X}$ in $\hat{Y}$ is trivial. Arguing as in the proof of Theorem \ref{thm2}, we can find a submanifold $\Sigma^k$ with nonzero intersection number with $\hat{X}$ with $Sc_n^{\rtimes}(\Sigma)>\frac{3}{2}$. Similarly we get a $\mu$-bubble $\Gamma^{k-1}$, far away from $\hat{X}$ and have $Sc_n^{\rtimes}(\Gamma)>1$. By the contractibility of $\hat{X}$ and \eqref{eq: H-isomorphism}, we have $H_i(\Tilde{Y})=0$ for $i\le k-1$. Therefore, by Lemma \ref{lem: HZ23} and Lemma \ref{lem: slice and dice}, we can fill $\Gamma^{k-1}$ by a $k$-chain in a neighbourhood of $\Gamma^{k-1}$ away from $\hat{X}$. This provides us with a closed $k$-chain with nonzero intersection number with $\hat{X}$, and a contradiction follows from Lemma \ref{lem: intersection 1}.
	\end{proof}
	
	\begin{proof}[Proof of Corollary \ref{cor5}]
		Let $G$ be the subgroup of $\pi_1(Y)$ and $X$ the $n-4$-dimensional closed aspherical manifold with $\pi_1(X)=G$. Let $J:\pi_1(X)\cong G\hookrightarrow \pi_1(Y)$ be a homomorphism. By Theorem 1B.9 in \cite{Hat02}, $J$ could be realized as a continuous map $f$. Since $n>2(n-4)$ when $n\le 7$, we can make $f$ into an embedding. The conclusion then follows from Corollary \ref{cor4}.
	\end{proof}

\end{document}